\documentclass[11pt]{amsart} \textwidth=14.5cm \oddsidemargin=1cm
\usepackage{amsmath,wasysym}
\usepackage{upgreek}
\usepackage{amsxtra}
\usepackage{amscd}
\usepackage{amsthm}
\usepackage{amsfonts}
\usepackage{amssymb}
\usepackage{dynkin-diagrams}
\usepackage{eucal}
\usepackage{graphics}
\usepackage{tikz}
\usepackage{tikz-cd}
\usepackage{stmaryrd}
\usepackage[all]{xy}
\usepackage{hyperref}
\usepackage{bbm} % \mathbbm
\allowdisplaybreaks

\hypersetup{colorlinks,linkcolor=blue, urlcolor=blue,
citecolor=blue}

\newtheorem{thm}{Theorem} [subsection]
\newtheorem{lem}[thm]{Lemma}
\newtheorem{cor}[thm]{Corollary}
\newtheorem{prop}[thm]{Proposition}

\theoremstyle{definition}
\newtheorem{definition}[thm]{Definition}
\newtheorem{example}[thm]{Example}

\newtheorem{rem}[thm]{Remark}

\numberwithin{equation}{subsection}

\newcommand{\nc}{\newcommand}
\nc{\browntext}[1]{\textcolor{brown}{#1}}
\nc{\greentext}[1]{\textcolor{green}{#1}}
\nc{\redtext}[1]{\textcolor{red}{#1}}
\nc{\bluetext}[1]{\textcolor{blue}{#1}}
\nc{\brown}[1]{\browntext{ #1}}
\nc{\green}[1]{\greentext{ #1}}
\nc{\red}[1]{\redtext{ #1}}
\nc{\blue}[1]{\bluetext{ #1}}

\newcommand{\qbi}{\left[\begin{array}{c}
4\\2
\end{array}\right]}
\newcommand{\Zq}{\mathbb{Z}[q, q^{-1}]}
\title[Lyndon bases of split $\imath$quantum groups]
{Lyndon bases of split $\imath$quantum groups}

\author[Run-Qiang Jian]{Run-Qiang Jian}
\address{Department of mathematics, Dongguan University of Technology, Dongguan 523808, China}
\email{jianrq@dgut.edu.cn (Jian)}

\author[Li Luo]{Li Luo}

\author[Xianfa Wu]{Xianfa Wu}

\address{School of Mathematical Sciences, Key Laboratory of MEA (Ministry of Education) \& Shanghai Key Laboratory of PMMP,  East China Normal University, Shanghai 200241, China}
\email{lluo@math.ecnu.edu.cn (Luo), 52275500051@stu.ecnu.edu.cn (Wu)}



\begin{document}

\begin{abstract}
We introduce and study Lyndon bases of the split $\imath$quantum groups $\mathbf{U}^\imath(\mathfrak{g})$. Especially, we provide the relation between Lyndon bases and Lusztig's PBW-type bases of $\mathbf{U}^\imath(\mathfrak{g})$, as well as a construction of canonical bases of $\mathbf{U}^\imath(\mathfrak{g})$ under an integral condition.
\end{abstract}

\maketitle
\setcounter{tocdepth}{1}
%\tableofcontents

%%%%%%%%%%%
%%%%%%%%%%%
\section{Introduction}
%==============================================
An $\imath$quantum group is a coideal subalgebra
$\mathbf{U}^\imath(\mathfrak{g})$ of a Drinfeld-Jimbo quantum
group $\mathbf{U}(\mathfrak{g})$ such that
$(\mathbf{U}(\mathfrak{g}), \mathbf{U}^\imath(\mathfrak{g}))$
forms a quantum symmetric pair in the sense of Letzter
\cite{Let99,Let02}. The classification of quantum symmetric pairs
(and hence $\imath$quantum groups) is formulated by Satake
diagrams, where a Satake diagram $(\mathbb{I}=\mathbb{I}_\circ\cup
\mathbb{I}_\bullet,\tau)$ consists of a bicolored partition of the
Dynkin diagram $\mathbb{I}=\mathbb{I}_\circ\cup
\mathbb{I}_\bullet$ and a (possibly trivial) diagram involution
$\tau$ subject to some compatibility conditions. The associated
$\imath$quantum groups are called quasi-split if
$\mathbb{I}_\bullet=\emptyset$; and are called split if it further
satisfies $\tau=\mathrm{id}$. Thanks to their pioneering work
\cite{BW18a} indicating the importance of $\imath$quantum groups in Kazhdan-Luszitg theory, Bao and Wang suggested to generalize
fundamental constructions for quantum groups to $\imath$quantum
groups.

In the classical theory of quantum groups, there are several kinds of bases, such as PBW-type bases \cite{Lu93}, canonical bases \cite{Lu90}, and crystal bases \cite{Ka91}. They play essential roles in the representation theory of quantum groups. It is a natural question that if such kinds of bases can be constructed in the framework of $\imath$quantum groups.
The study of canonical bases of modified $\imath$quantum groups (named $\imath$canonical bases) was initiated in \cite{BW18a} and improved in \cite{BW18b, BW21} by Bao and Wang. In addition, there are completely different ``dual $\imath$canonical bases'' of universal $\imath$quantum groups via a quiver variety construction by Lu and Wang \cite{LW21b}. The crystal bases of $\imath$quantum groups were investigated by Watanabe for some quasi-split types \cite{Wat20, Wat24}.

The first work on PBW-type bases of $\imath$quantum groups is Iorgov and Klimyk's investigation \cite{IK00} for the special split $\imath$quantum group associated with the type $A_{n}$ Dynkin diagram, which is the non-standard $q$-deformation of $U(\mathfrak{so}_{n+1})$ introduced in \cite{GK91}. Xu and Yang \cite{XY14} explicitly constructed PBW-type bases for all split $\imath$quantum groups by using the braid group actions given in \cite{KP11}. Lu and Wang \cite{LW22} provided PBW-type bases of quasi-split $\imath$quantum groups associated with simply laced type ADE Dynkin digrams via an $\imath$Hall algebra approach originating from Ringel's construction \cite{Rin96}, and further described these PBW-type bases more explicitly in \cite{LW21a} by using the braid group actions.

On the other hand, in their work \cite{LR95}, Lalonde and Ram constructed a new type of Gr\"{o}bner basis, which are called Lyndon bases nowadays, for universal enveloping algebras of Lie algebras in terms of good words and Lyndon words. Based on his quantum shuffle approach to quantum groups (see \cite{Ro98}), Rosso generalized this construction of Lyndon bases to the positive part of quantum groups in an unpublished paper \cite{Ro02}. As applications, he derived a multiplicative formula for the universal $R$-matrix as a product of $q$-exponentials, and provided a new proof of the existence of canonical bases. In \cite{Lec04}, Leclerc showed that Rosso's Lyndon bases are proportional to Lusztig's PBW-type bases obtained by braid group actions, and used this fact to provide another proof of the existence of dual canonical bases. A direct shuffle construction of dual canonical bases (in a super generalization) is due to Clark, Hill and Wang \cite{CHW16}.

Contrary to those constructions of bases mentioned before, Lyndon bases are much easier to compute. However, there has been no work to give bases for $\imath$quantum groups by using good Lyndon words yet. In this paper, we aim at generalizing the construction in \cite{LR95} to split $\imath$quantum groups. We introduce good words for any split $\imath$quantum group $\mathbf{U}^\imath(\mathfrak{g})$ in the sense of \cite{LR95}, and give the good Lyndon words by which we can construct a Lyndon basis of $\mathbf{U}^\imath(\mathfrak{g})$ (see Theorem~\ref{lyndonbasis}). Furthermore, we provide a relationship between the Lyndon basis and Xu-Yang's PBW-type basis in Theorem~\ref{thm:pbwlyndon}, which says that each Lyndon basis element is the leading term of a PBW-type basis element in some sense. In particular, we prove in Corollary~\ref{cor:lynandpbw} that the Lyndon basis coincides with Xu-Yang's PBW-type basis for the type A split $\imath$quantum group $\mathbf{U}^\imath(\mathfrak{sl}_n)$.

These Lyndon bases enable us to construct canonical bases for $\mathbf{U}^\imath(\mathfrak{g})$, as Rosso \cite{Ro02} and Leclerc \cite{Lec04} did for quantum groups. However, their approach does not work in our case due to the lack of quantum shuffle realization of $\imath$quantum groups at this stage. Instead, we introduce an integral condition by which a canonical basis can be produced (Theorem~\ref{Canonical bases}). As an application, in Theorem~\ref{canofA} we show that the split $\imath$quantum groups $\mathbf{U}^\imath(\mathfrak{sl}_n)$ satisfy this integral condition and hence afford canonical bases. We believe that the integral condition would have further interests; but on the other hand, we must also acknowledge that it is not always valid (see Example~\ref{ex:G2} for type $G_2$).

The paper is organized as follows. In Section 2, we recall the basic definitions of quantum groups and split $\imath$quantum groups. In Section 3, we introduce good Lyndon words and study some of their properties. We construct Lyndon bases for split $\imath$quantum groups in Section 4. We also study the relationship between Lyndon bases and PBW-type bases. Finally, we establish canonical bases for the split $\imath$quantum groups $\mathbf{U}^\imath(\mathfrak{sl}_n)$ in Section 5.

\subsubsection*{Acknowledgement}
We are grateful to Marc Rosso for sharing his unpublished paper \cite{Ro02} with us. We thank Weiqiang Wang for confirming the reasonableness of our canonical bases and informing the work \cite{LW21b}. Finally, we thank the referee for his/her careful reading and suggestions. LL is partially supported by the National Key R\&D Program of China (No. 2024YFA1013802) and the NSF of China (No. 12371028).

%=======================================================================================
\section{Quantum groups and split $\imath$quantum groups}
%=======================================================================================
In this section, we recall the notions of quantum groups and split $\imath$quantum groups. We start by fixing some notations that will be used in the sequel. Let $\mathfrak{g}$ be a complex simple Lie algebra of rank $n$.
We fix a set of simple roots $\Pi=\{\alpha_1,\alpha_2,\ldots,\alpha_n\}$ of $\mathfrak{g}$. Let $\Phi^+$ be the set of positive roots, and $(\cdot, \cdot)$ a symmetric bilinear form on the root lattice $\mathbb{Z}\Pi$.
Denote by $A=(a_{ij})_{1\leq i,j\leq n}$ the corresponding Cartan matrix, i.e., $a_{ij}=\frac{2(\alpha_i,\alpha_j)}{(\alpha_i,\alpha_i)}$. Write $d_i=\frac{(\alpha_i,\alpha_i)}{2}$.
Let $W$ be the Weyl group of $\mathfrak{g}$, which is generated by the simple reflections $s_i=s_{\alpha_i}$ ($1\leq i\leq n$).

%---------------------------------------------------------
\subsection{Quantum groups}
%---------------------------------------------------------

Suppose that $q\in \mathbb{C}$ is not a root of unity and set $q_i=q^{d_i}$ for any $1\leq i\leq n$.
For any positive integer $k$, the quantum integers, quantum factorials, and quantum binomial coefficients are defined respectively by $$[k]_i=\frac{q_i^k-q_i^{-k}}{q_i-q_i^{-1}},\qquad [k]_i!=[1]_i[2]_i\cdots [k]_i, \qquad \left[\begin{array}{c}
m\\k
\end{array}\right]_i=\frac{[m]_i[m-1]_i\cdots[m-k+1]_i}{[k]_i!},$$
where $m\in\mathbb{N}$ with $k\leq m$. If $d_i=1$, we shall write $[k]=[k]_i$ and $[k]!=[k]_i!$.

The quantum group $\mathbf{U} = U_q(\mathfrak{g})$ is the $\mathbb{Q}(q)$-algebra generated by $E_i, F_i, K_i$ and $K_i^{-1}$ for $1 \le i \le n$ subject to the following relations: for $1 \le i, j \le n$,
\begin{align*}
&K_i K_i^{-1} =1 = K_i^{-1}K_i, \quad K_iK_j=K_jK_i,\\
 &K_iE_jK_i^{-1} = q_i^{a_{ij}} E_j, \quad K_iF_jK_i^{-1} = q_i^{-a_{ij}} F_j,\\
  &E_iF_j - F_jE_i = \delta_{ij}\frac{K_i-K_i^{-1}}{q_i-q_i^{-1}},\\
  &\sum_{k=0}^{1-a_{ij}} (-1)^k\begin{bmatrix}1-a_{ij}\\k\end{bmatrix}_i E_i^{1-a_{ij}-k}E_jE_i^k = 0, \quad (i \neq j),\\
  &\sum_{k=0}^{1-a_{ij}} (-1)^k\begin{bmatrix}1-a_{ij}\\k\end{bmatrix}_i F_i^{1-a_{ij}-k}F_jF_i^k = 0, \quad (i \neq j).
\end{align*}

As usual, we denote by $\mathbf{U}^+$ (resp. $\mathbf{U}^-$) the subalgebra of $\mathbf{U}$ generated by $E_i$ (resp. $F_i$) for $1\le i \le n$.

%---------------------------------------------------------
\subsection{Split $\imath$quantum groups}
%---------------------------------------------------------
The split $\imath$quantum group $\mathbf{U}^{\imath}=\mathbf{U}^\imath(\mathfrak{g})$ is the $\mathbb{Q}(q)$-subalgebra of $\mathbf{U}$ generated by $B_i = F_i + \xi_i E_iK_i^{-1} (1 \le i \le n)$,
where $\{\xi_i\}_{1\le i \le n}$ are parameters in $\mathbb{Q}(q)^{\times}$.
It was shown in \cite{Let02} and \cite[Lemma~8.6]{LW21a} that the above algebras with different choice of $\{\xi_i\}_{1\le i \le n}$ are isomorphic.
By \cite{Ko14} and \cite{CLW21}, these generators $B_i$ ($1\leq i\leq n$) satisfy the following relations: for $1\leq i\neq j\leq n$,
\begin{align*}
  &\sum_{k=0}^{1-a_{ij}}(-1)^k\begin{bmatrix}1-a_{ij}\\k\end{bmatrix}_i B_i^{1-a_{ij}-k}B_jB_i^k \\
  &= \begin{cases}
    0,&\text{if }a_{ij}=0,\\[3pt]
    q_i \xi_iB_j,&\text{if }a_{ij}=-1,\\[3pt]
    [2]_i^2 q_i \xi_i(B_iB_j-B_jB_i),&\text{if }a_{ij}=-2,\\[3pt]
    ([3]_i^2+1) q_i \xi_i(B_i^2B_j+B_jB_i^2)&\\[3pt]
    \ \ \ - [2]_i([2]_i[4]_i+q_i^2+q_i^{-2})q_i \xi_i B_iB_jB_i - [3]_i^2 (q_i \xi_i)^2 B_j,&\text{if }a_{ij}=-3.
  \end{cases}
\end{align*}

Let $V$ be a vector space over $\mathbb{Q}(q)$ with a fixed basis $v_1,v_2,\ldots, v_n$. Denote by $T(V)$ the tensor algebra built on $V$. For any $1\leq i\neq j\leq n$, we set\begin{align*}
h_{ij}  &= \begin{cases}
    0,&\text{if }a_{ij}=0,\\[3pt]
    q_i \xi_iv_j,&\text{if }a_{ij}=-1,\\[3pt]
    [2]_i^2 q_i \xi_i(v_i\otimes v_j-v_j\otimes v_i),&\text{if }a_{ij}=-2,\\[3pt]
    ([3]_i^2+1) q_i \xi_i(v_i\otimes v_i\otimes v_j+v_j\otimes v_i\otimes v_i)&\\[3pt]
    \ \ \ - [2]_i([2]_i[4]_i+q_i^2+q_i^{-2})q_i \xi_i v_i\otimes v_j\otimes v_i - [3]_i^2 (q_i \xi_i)^2 v_j,&\text{if }a_{ij}=-3.
  \end{cases}
\end{align*}
Let $I$ be the two-sided ideal generated by elements $$\sum_{k=0}^{1-a_{ij}}(-1)^k\begin{bmatrix}1-a_{ij}\\k\end{bmatrix}_i v_i^{\otimes(1-a_{ij}-k)}\otimes v_j\otimes v_i^{\otimes k}-h_{ij},\quad i\neq j.$$Then the quotient algebra $T(V)/I$ is isomorphic to $\mathbf{U}^{\imath}$. Denote by $\pi^{\imath}:T(V) \to \mathbf{U}^{\imath}$ the homomorphism satisfying $\pi^{\imath}(v_i) = B_i$.

Consider the filtration of $T(V)$ given by $T_m(V) = \bigoplus_{k=0}^m V^{\otimes k}$ with $m\ge 0$. Let $\mathbf{U}^{\imath}_m = \pi^{\imath}\big(T_m(V)\big)$ for $ m \ge 0$ and $ \mathbf{U}^{\imath}_{-1} = 0$.
Then we have
$\mathbf{U}^{\imath}_m \subset \mathbf{U}^{\imath}_{m+1}$ and $\mathbf{U}^{\imath}_m \mathbf{U}^{\imath}_l \subset \mathbf{U}^{\imath}_{m+l}$. Hence $\mathbf{U}^{\imath}_m$ ($m\ge0$) provide a filtration on $\mathbf{U}^{\imath}$.
Let \[
  \text{gr}(\mathbf{U}^{\imath}) = \bigoplus_{m=0}^{\infty} \mathrm{gr}_m(\mathbf{U}^{\imath})
\] be the associated graded algebra. Here
\[ \text{gr}_0(\mathbf{U}^{\imath}) = \mathbf{U}^{\imath}_0 \quad\mbox{and}\quad
  \text{gr}_m(\mathbf{U}^{\imath}) = \mathbf{U}^{\imath}_m/\mathbf{U}^{\imath}_{m-1}\quad (\text{for } m>0).
\] One can define a surjective algebra homomorphism $\varphi^-: \mathbf{U}^- \rightarrow \mathrm{gr}(\mathbf{U}^{\imath})$ which sends $F_i$ to the image of $B_i$ in $\mathrm{gr}(\mathbf{U}^{\imath})$. In general, $$\varphi^-(F_{i_1}F_{i_2}\cdots F_{i_k})=B_{i_1}B_{i_2}\cdots B_{i_k}+\mathbf{U}^{\imath}_{k-1},\quad k\geq 1.$$For each nonnegative integer $m$, denote by $\mathbf{U}^{-}_m$ the subspace of $\mathbf{U}^-$ spanned by elements of the form $F_{i_1}F_{i_2}\cdots F_{i_m}$. Obviously, one has $\dim \mathbf{U}^{-}_m\geq \dim \mathrm{gr}_m(\mathbf{U}^{\imath})$. By an argument in \cite{Let02} (see also the proof of \cite[Theorem~4.4]{XY14}), $\dim \mathbf{U}^{-}_m\leq \dim \mathrm{gr}_m(\mathbf{U}^{\imath})$. So $\dim \mathbf{U}^{-}_m= \dim \mathrm{gr}_m(\mathbf{U}^{\imath})$ for each $m$, and hence $\varphi$ is an isomorphism.

Similarly, define $\varphi^+: \mathbf{U}^+ \rightarrow \mathrm{gr}(\mathbf{U}^{\imath})$ to be the algebra isomorphism sending $E_i$ to the image of $B_i$ in $\mathrm{gr}(\mathbf{U}^{\imath})$.

%=======================================================================================
\section{ Lyndon words and good words}
%=======================================================================================
In this section, we will introduce the notion of $\mathbf{U}^\imath$-good word. We will always assume that $V$ is a vector space over $\mathbb{Q}(q)$ with a preferred basis $v_1, v_2,\ldots,v_n$. For simplicity, we always write a pure tensor $v_{i_1}\otimes \cdots \otimes v_{i_k}$ in $T(V)$ as $v_{i_1} \cdots  v_{i_k}$. We call the basis elements $v_i$ \emph{letters}, and call a pure tensor in $T(V)$ a \emph{word}. The set of words is denote by $\mathcal{W}$. The set including all words and the empty word is denoted by $\mathcal{W}^\ast$. Therefore, $\mathcal{W}^\ast$ is a basis of $T(V)$. For a word $w=v_{i_1} \cdots  v_{i_k}$, the number $k$ is called the \emph{length} of $w$, and is denoted by $\ell(w)$. We assign each word $x=v_{i_1} \cdots  v_{i_k}$ an $\mathbb{N}\Pi$-degree $|x|=\alpha_{i_1}+\cdots+\alpha_{i_k}$. Given $\nu\in \mathbb{Q}(q)$, the $\nu$-commutator $[\cdot,\cdot]_\nu$ on $T(V)$ is defined to be $$[x,y]_\nu=xy-\nu^{(|x|,|y|)}yx,$$ for any $x,y\in \mathcal{W}$.

%---------------------------------------------------------
\subsection{Lyndon words}
%---------------------------------------------------------
We endow the set $\mathcal{W}$ with the following lexicographic order $<$:
\begin{itemize}
\item[(1)] $v_1<v_2<\cdots<v_n$;
\item[(2)] for any $x,y\in \mathcal{W}$, $x<y$ if either $y=xz$ for some $z\in \mathcal{W}$ of length $\geq 1$, or $x=uv_iw_1$ and $y=uv_jw_2$ with $v_i<v_j$ and $u,w_1,w_2\in \mathcal{W}^*$.
\end{itemize}

\begin{definition}
  A word $l=v_{i_1}\cdots v_{i_k}$ is called a \emph{Lyndon word} if either $\ell(l)=1$ or
  $$l<v_{i_j}v_{i_{j+1}}\cdots v_{i_k}\quad\mbox{for all $j=2,\ldots,k$.}$$
  We denote by $\mathcal{L}$ the set of Lyndon words.
\end{definition}

\begin{prop}\label{Basic facts of Lyndon words}
\begin{itemize}  \item[(1)] A word $l$ is Lyndon if and only if either $l=v_i$ for some $i$, or $l=l_1l_2$ for some $l_1,l_2\in \mathcal{L}$ with $l_1<l_2$.
\item[(2)] Any word $w$ can be written uniquely in the form $$w=l_1l_2\cdots l_m$$for some $l_1,\ldots,l_m\in \mathcal{L}$ with $l_1\geq l_2\geq \cdots \geq l_m$.
\end{itemize}
\end{prop}
\begin{proof}See e.g., \cite[Proposition~5.1.3 \& Theorem~5.1.5]{Lo97}.\end{proof}

\begin{rem}Given $l\in \mathcal{L}$ with $\ell(l)>1$, by the above proposition, one can write $l=l_1l_2$, where $l_1$ is the left proper Lyndon factor of $l$ with maximal length. Then $l_1\in \mathcal{L}$ and $l_1<l_2$. We call the pair $(l_1,l_2)$ the \emph{co-standard factorization} of $l$, and denote it by $\mu(l)$.\end{rem}

Given a Lyndon word $l\in \mathcal{L}$, we define the \emph{$q^{-1}$-bracketing} $[l]$ of $l$ inductively by $[l]=l$
if $l$ is a letter, and $$[l] = [[l_1],[l_2]]_{q^{-1}}\in T(V)$$ if $\ell(l) > 1$ and $\mu(l)=(l_1,l_2)$.
For any $w\in \mathcal{W}$, if $w=l_1l_2\cdots l_m$ for some $l_1,\ldots,l_m\in \mathcal{L}$ with $l_1\geq l_2\geq \cdots \geq l_m$,
then we define $[w]=[l_1][l_2]\cdots [l_m]$. We also need the notion of \emph{$q$-bracketing} $\{w\}$ of $w$ which replaces the $q^{-1}$-commutator by the $q$-commutator in the above definition.

%---------------------------------------------------------
\subsection{Good words}
%---------------------------------------------------------
Following \cite{LR95}, we define another total order $\preceq$ on $\mathcal{W}$ as follows.
For any $x,y\in \mathcal{W}$, $x\prec y$ if either the length of $x$ is strictly less than that of $y$,
or they have the same length and $x> y$.
Notice that there are only finitely many words less than a given word with respect to this order. So $(\mathcal{W},\prec)$ is well ordered. We can apply the well-founded induction to $(\mathcal{W},\prec)$.

 The following properties of the order $\prec$ are useful in our study of Lyndon bases.

\begin{prop}\label{Basic facts for LR order}Suppose that $x,y\in \mathcal{W}$ satisfy $x\prec y$.\begin{itemize}
\item[(1)] For any $w,z\in \mathcal{W}^\ast$, $wxz\prec wyz$.
\item[(2)] If $\ell(x)=\ell(y)$, then $x w\prec yz$ for any $w,z\in \mathcal{W}^\ast$ with $\ell(w)=\ell(z)$.
\end{itemize}\end{prop}
\begin{proof}Observe that if $\ell(x)<\ell(y)$, then $\ell(wxz)<\ell(wyz)$.

If $\ell(x)=\ell(y)$, then $x=uv_ix'$ and $y=uv_jy'$ for some $u,x',y'\in \mathcal{W}$ and $1\leq j<i\leq n$. Obviously one has $$wuv_ix'z>wuv_jy'z$$ and $$uv_ix'w>uv_jy'z.$$\end{proof}

For any word $w=v_{i_1}v_{i_2}\cdots v_{i_k}$, we write $B_w=B_{i_1}B_{i_2}\cdots B_{i_k}$ in $\mathbf{U}^\imath$.
If $x=\sum a_iw_i\in T(V)$ with $a_i\in \mathbb{Q}(q)$ and $w_i\in \mathcal{W}$,
we write $B_x=\sum a_iB_{w_i}$. Similar conventions are adopted for elements in $\mathbf{U}^+$ and $\mathbf{U}^-$.

\begin{definition} A word $g\in \mathcal{W}$ is said to be $\mathbf{U}^\imath$-\emph{good} if $B_g\in \mathbf{U}^\imath$ can not
  be expressed as a linear combination of $B_w\in \mathbf{U}^\imath$ for $w\in \mathcal{W}$ with $w\prec g$. We denote by $\mathcal{G}(\mathbf{U}^\imath)$ the set of $\mathbf{U}^\imath$-good words.
\end{definition}

\begin{rem}A word $w$ is not $\mathbf{U}^\imath$-good if and only if $$ B_w - \sum_{ \ell(u) = \ell(w),u > w}a_{u,w} B_u \in\mathbf{U}^{\imath}_{\ell(w)-1},$$ where $a_{u,w}\in \mathbb{Q}(q)$.
In fact, a word $w$ is not $\mathbf{U}^\imath$-good if and only if $$B_w=\sum_{ \ell(u) = \ell(w),u > w}a_{u,w} B_u+\sum_{ \ell(v) < \ell(w)}a_{v,w} B_v,$$where $a_{u,w},a_{v,w}\in \mathbb{Q}(q)$ and some of them are nonzero. Moreover, since $B_x\notin \mathbf{U}^{\imath}_{\ell(x)-1}$ for all $x\in \mathcal{W}$, any of the above expressions of $B_w$ with $w\notin \mathcal{G}(\mathbf{U}^\imath)$ must imply that some $a_{u,w}$ should be nonzero.\end{rem}

Recall that a word $g$ is $\mathbf{U}^+$-good if $E_g\in \mathbf{U}^+$ cannot be expressed as a linear combination of $E_w\in \mathbf{U}^+$ for $w\in \mathcal{W}$ with $w> g$ (for more details, see \cite{Lec04}).

\begin{prop}\label{Equivalence of two goodnesses}
  A word is $\mathbf{U}^{\imath}$-good if and only if it is $\mathbf{U}^+$-good.
\end{prop}
\begin{proof} We will show the equivalent statement that a word is not $\mathbf{U}^{\imath}$-good if and only if it is not $\mathbf{U}^+$-good.

If $w$ is not $\mathbf{U}^\imath$-good, then \begin{align*}
E_w&=(\varphi^+)^{-1}\big(B_w+\mathbf{U}^{\imath}_{\ell(w)-1}\big)\\[3pt]
&=(\varphi^+)^{-1}\bigg(\sum_{ \ell(u) = \ell(w),u > w}a_{u,w} B_u+\mathbf{U}^{\imath}_{\ell(w)-1}\bigg)\\[3pt]
&=\sum_{ \ell(u) = \ell(w),u > w}a_{u,w} E_u.
\end{align*}
Here, $\varphi^+$ is the algebra isomorphism from $\mathbf{U}^+$ to $\text{gr}(\mathbf{U}^{\imath})$ given in the last section. So $w$ is not $\mathbf{U}^+$-good.

Conversely, if $g$ is not $\mathbf{U}^+$-good, then \[
        E_w = \sum_{\ell(u) = \ell(w), u > w}a_{u,w} E_u,\] with $a_u\in \mathbb{Q}(q)$. Therefore,
  \begin{align*}
B_w+\mathbf{U}^{\imath}_{\ell(w)-1}&=\varphi^+( E_w)\\[3pt]
&=\varphi^+\bigg(\sum_{ \ell(u)=\ell(w),u > w}a_{u,w} E_u \bigg)\\[3pt]
&=\sum_{\ell(u) =\ell(w),u > w}a_{u,w} B_u+\mathbf{U}^{\imath}_{\ell(w)-1},
\end{align*}which means $w$ is not $\mathbf{U}^\imath$-good.\end{proof}

We denote $\mathcal{GL}(\mathbf{U}^\imath)=\mathcal{G}(\mathbf{U}^\imath)\cap \mathcal{L}$, the set of $\mathbf{U}^\imath$-good Lyndon words.

\begin{rem}\label{Complete list of good Lyndon words}
In \cite[Section 8]{Lec04}, Leclerc gave a complete description of the sets of $\mathbf{U}^+$-good Lyndon words. By Proposition \ref{Equivalence of two goodnesses}, one has a complete description of $\mathcal{GL}(\mathbf{U}^\imath)$.\end{rem}

The following property of $\mathbf{U}^\imath$-good words is a direct modification of the one for classical good words (see \cite[Proposition~2.4]{LR95}).
For the completeness of the paper, we provide its proof here.

\begin{prop}\label{Factors of good words}
  Every factor of a $\mathbf{U}^\imath$-good word is $\mathbf{U}^\imath$-good.
\end{prop}
\begin{proof}
  Suppose that $x,y,z\in \mathcal{W}^\ast$ satisfy $xyz\in \mathcal{G}(\mathbf{U}^\imath)$ but $y\notin \mathcal{G}(\mathbf{U}^\imath)$.
Thus $$B_y=\sum_{u\prec y}a_{u,y}B_u,$$ with $a_{u,y}\in \mathbb{Q}(q)$. Since $u\prec y$, we have $xuz\prec xyz$. So
  \begin{align*}
    B_{xyz}=B_xB_yB_z
    =B_x\bigg(\sum_{u\prec y}a_{u,y}B_u\bigg)B_z
    =\sum_{u\prec y}a_{u,y} B_x B_u B_z
    =\sum_{u\prec y}a_{u,y} B_{xuz},
  \end{align*}
and hence $xyz\notin \mathcal{G}(\mathbf{U}^\imath)$, a contradiction.
\end{proof}

The following result follows immediately from Propositions \ref{Basic facts of Lyndon words}
 \& \ref{Factors of good words}.

\begin{prop}\label{Decomposition of good Lyndon words}
  Any $\mathbf{U}^\imath$-good word $g$ can be written uniquely in the form $$g=l_1l_2\cdots l_m$$for some $l_1,\ldots,l_m\in \mathcal{GL}(\mathbf{U}^\imath)$ with $l_1\geq l_2\geq \cdots \geq l_m$.
\end{prop}

%---------------------------------------------------------
\section{Lyndon bases of $\imath$quantum groups}
%---------------------------------------------------------
In this section, we will construct the Lyndon basis of $\mathbf{U}^\imath$. We keep all assumptions in the above sections.
\subsection{Lyndon bases of $\mathbf{U}^\imath$}

\begin{thm}\label{Good words bases for iquantum}
  The set $$\mathbf{B}_{\mathcal{G}}=\{B_g~|~g\in \mathcal{G}(\mathbf{U}^\imath)\}$$ forms a basis of $\mathbf{U}^\imath$.
\end{thm}
\begin{proof}
  We first show that $\mathbf{U}^\imath$ is spanned by the set $\{B_g~|~g\in \mathcal{G}(\mathbf{U}^\imath)\}$.
It suffices to show that for any word $w$ the element $B_w$ is a linear combination of $B_g$ with $g\in \mathcal{G}(\mathbf{U}^\imath)$.
If $w\in \mathcal{G}(\mathbf{U}^\imath)$, we are done. If not, then $B_w=\sum_{x\prec w}a_{x,w}B_x$ for some $a_{x,w}\in \mathbb{Q}(q)$.
If all $x$ appearing in the sum are good, we are done.
If there exists a word $x\prec w$ such that $x\notin \mathcal{G}(\mathbf{U}^\imath)$, we can express $B_x$ as a linear combination of $B_y$ with $y \prec x$.
Since the number of words strictly smaller than $w$ with respect to $\prec$ is finite, we get the desired result.

  Now assume that $\sum_{g\in \mathcal{G}(\mathbf{U}^\imath)} a_g B_g=0$ with $a_g=0$ for all but finitely many $g$.
If there exist some $g\in \mathcal{G}(\mathbf{U}^\imath)$ such that $a_g\neq 0$, then choose the maximal one $g_0$ among these elements.
Then one can express $B_{g_0}$ as a linear combination of $B_g$ with $g \prec g_0$, a contradiction.
\end{proof}

\begin{rem}In the proof of the above theorem, one sees that the expression of an element $B_w\in \mathbf{U}^\imath$ with $w\in \mathcal{W}$ in terms of $\mathbf{B}_{\mathcal{G}}$ is of the form $$B_w=\sum_{g\in \mathcal{G}(\mathbf{U}^\imath), g\preceq w}a_{g,w}B_g,$$where $a_{g,w}\in \mathbb{Q} (q)$.\end{rem}

\begin{lem}\label{bracket triangular}
  For any $l\in\mathcal{L}$, we have $$[l]=l+\sum_{\ell(w)=\ell(l),w>l}b_{w,l} w$$ for some $b_{w,l}\in \mathbb{Z}[q,q^{-1}]$.
\end{lem}
\begin{proof}It is an easy induction on $\ell(l)$. For more details, see \cite[Prop.~19]{Lec04}.\end{proof}

In fact, this lemma can be generalized to arbitrary words.

\begin{prop}\label{bracket triangular for words}For any $w\in \mathcal{W}$, we have $$[w]=w+\sum_{\ell(x)=\ell(w),x>w}b_{x,w}x,$$where $b_{x,w}\in \Zq$.
\end{prop}
\begin{proof}Write $w=l_1l_2\cdots l_k$ with $l_1,l_2,\dots,l_k\in \mathcal{L}$ and $l_1\geq l_2\geq \cdots\geq l_k$. We use induction on $k$. If $k=1$, it is just above lemma. Assume that $k>1$ and the result holds for $k-1$. Set $w'=l_2\cdots l_k$. Then by the inductive hypothesis, $$[w']=w'+\sum_{\ell(y)=\ell(w'),y>w'}b_{y,w'}y$$where $b_{y,w'}\in \Zq$. We also write $$[l_1]=l_1+\sum_{\ell(x)=\ell(l_1),x>l_1}b_{x,l_1}x$$for some $b_{x,l_1}\in \Zq$. Hence \begin{align*}
[w]&=[l_1][w']\\
&=\bigg(l_1+\sum_{\ell(x)=\ell(l_1),x>l_1}b_{x,l_1}x\bigg)\bigg(w'+\sum_{\ell(y)=\ell(w'),y>w'}b_{y,w'}y\bigg)\\
&=l_1 w'+\sum_{\ell(x)=\ell(l_1),x>l_1}b_{x,l_1}x w'+\sum_{\ell(y)=\ell(w'),y>w'}b_{y,w'}l_1y\\
&\ \ \ \ +\sum_{\ell(x)=\ell(l_1),x>l_1}\sum_{\ell(y)=\ell(w'),y>w'}b_{x,l_1}b_{y,w'}x y.
\end{align*}By Proposition \ref{Basic facts for LR order}, one has $xw',l_1y,xy\prec w$. Since $\ell(xw')=\ell(l_1y)=\ell(xy)=\ell(w)$, we have $xw',l_1y,xy> w$.\end{proof}

\begin{thm}\label{lyndonbasis}
Each of the following sets is a basis of $\mathbf{U}^{\imath}$:
\begin{itemize}
\item[(1)] $\mathbf{B}_{\mathcal{GL}}=\{B_{l_1l_2\cdots l_m}~|~l_1,\ldots,l_m\in \mathcal{GL}(\mathbf{U}^{\imath})\text{ and }l_1\geq \cdots\geq l_m\}$;
\item[(2)] $\mathbf{B}_{\mathcal{L}}=\{B_{[l_1][l_2]\cdots [l_m]}~|~l_1,\ldots,l_m\in \mathcal{GL}(\mathbf{U}^{\imath})\text{ and }l_1\geq \cdots\geq l_m\}$.
  \end{itemize}
\end{thm}
\begin{proof}
  For (1), we first note, from Proposition \ref{Decomposition of good Lyndon words} and Theorem~\ref{Good words bases for iquantum}, that $\mathbf{U}^{\imath}=\mathrm{Span}\mathbf{B}_{\mathcal{GL}}$. Assume that there exists a subset $S$ of $\mathcal{W}$ and nonzero numbers $\{a_w\}_{w\in S} \subset \mathbb{Q}(q)$ such that $B_w\in \mathbf{B}_{\mathcal{GL}}$ and $\sum_{w\in S} a_w B_w = 0$. Let $S_0 $ be the set of words with maximal length in $S$. Then
$\sum_{w\in S_0} a_w B_w+\mathbf{U}^{\imath}_{m-1}$ is zero in $\text{gr}_m(\mathbf{U}^{\imath})$, where $m$ is the length of words in $S_0$.
As a consequence, $$\sum_{w\in S_0} a_w E_w = (\varphi^+)^{-1}\bigg(\sum_{w\in S_0} a_w B_w+\mathbf{U}^{\imath}_{m-1}\bigg)=0.$$ But the set
  \[
    \{E_{l_1l_2\cdots l_m}~|~l_1,\ldots,l_m\in \mathcal{GL}(\mathbf{U}^+)\text{ and }l_1\geq \cdots\geq l_m\}
  \]
forms a basis of $\mathbf{U}^+$ (see Theorem 23 in \cite{Lec04}), a contradiction.

Statement (2) follows immediately from (1) and Lemma~\ref{bracket triangular}.
\end{proof}

\begin{rem}\label{Description of good words}By the above theorem, we have $$\mathcal{G}(\mathbf{U}^\imath)=\{l_1l_2\cdots l_m~|~l_1,\ldots,l_m\in \mathcal{GL}(\mathbf{U}^{\imath})\text{ and }l_1\geq \cdots\geq l_m\}.$$ Indeed, it follows from Proposition \ref{Decomposition of good Lyndon words} that $\mathcal{G}(\mathbf{U}^\imath)$ is contained in the set on the right hand side. Conversely, since both $\mathbf{B}_{\mathcal{G}}$ and $\mathbf{B}_{\mathcal{GL}}$ are bases of $\mathbf{U}^{\imath}$, we get the inclusion in the other direction.
\end{rem}

\begin{definition}\label{def:lyndon}
We call $\mathbf{B}_{\mathcal{L}}$ the \emph{Lyndon basis} of $\mathbf{U}^{\imath}$.
\end{definition}

\begin{example}[Lyndon basis of type $A_2$]By Lemma \ref{Complete list of good Lyndon words} and \cite[8.1]{Lec04}, the good Lyndon words in type $A_2$ are $v_1, v_1v_2, v_2$. We have\begin{align*}
    B_{[v_1]} &=  B_{v_1}, \\
    B_{[v_1v_2]} &= B_{ v_1v_2}-qB_{ v_2v_1}, \\
    B_{[v_2]} &= B_{v_2}.
\end{align*}\end{example}

\begin{example}[Lyndon basis of type $B_2$]\label{Lyndon basis of B2}By Lemma \ref{Complete list of good Lyndon words} and \cite[8.2]{Lec04}, the good Lyndon words are $v_1, v_1v_2, v_1v_2v_2, v_2$. We adopt the convention that $$(\alpha_1,\alpha_1)=4,\quad (\alpha_2,\alpha_2)=2,\quad(\alpha_1,\alpha_2)=-1.$$ Then\begin{align*}
    B_{[v_1]} &=  B_{v_1}, \\
    B_{[v_1v_2]} &= B_{ v_1v_2}-q^2B_{ v_2v_1}, \\
        B_{[v_1v_2v_2]} &= B_{ v_1v_2v_2}-(q^2+1)B_{ v_2v_1v_2}+q^2B_{v_2 v_2v_1}, \\
    B_{[v_2]} &= B_{v_2}.
\end{align*}\end{example}

\subsection{Relations with PBW-type bases} In this subsection, we fix $\xi_i = - q_i^{-2}$, for all $i$, in the definition of $B_i$.

In \cite{Lu93}, Lusztig defined a series of braid group actions on $\mathbf{U}$. For $1 \le i \le n$, let $T_i$ be the algebra automorphism on $\mathbf{U}$ defined below:
  \[
    T_i(E_i) = -K_i^{-1} F_i,\qquad  T_i(F_i) = -E_iK_i,\qquad  T_i(K_i) = K_i^{-1},
  \]
  and for $j \neq i$,
  \begin{align*}
    T_i(E_j) &= \sum_{s=0}^{-a_{ij}} (-1)^s q_i^{-s} E_i^{(s)}E_jE_i^{(-a_{ij}-s)},\quad
    T_i(F_j) = \sum_{s=0}^{-a_{ij}} (-1)^s q_i^{s} F_i^{(-a_{ij}-s)}F_jF_i^{(s)},\\
     T_i(K_j) &= K_j K_i^{-a_{ij}}.
  \end{align*}
The notation $T_i$ is just the one $T'_{i,-1}$ in \cite[37.1.3]{Lu93} with $v=q$. Our choice of braid group actions here coincides with those in \cite{KP11} and \cite{XY14}. Notice that the braid actions used in \cite{Lec04} are $T'_{i,-1}$'s with $v=q^{-1}$.

Denote by $\tau:\mathbf{U}\rightarrow\mathbf{U}$ the unique $\mathbb{Q}$-algebra antiautomorphism satisfying $$\tau(q)=q^{-1},\quad \tau{E_i}=F_i,\quad \tau (F_i)=E_i,\quad \tau(K_i)=K_i^{-1}.$$ Then one has $$\tau T_i=T_i\tau$$for each $i$.

Let $w_0=s_{i_1}s_{i_2}\cdots s_{i_N}$ be a reduced expression of the longest element $w_0\in W$. Denote $\beta_j = s_{i_1}s_{i_2}\cdots s_{i_{j-1}}(\alpha_j)$ for $1\leq j\leq N$.
Then $\Phi^+=\{\beta_1, \beta_2,\ldots, \beta_N\}$. Denote
\[
  F_{\beta_j}=T_{i_1}T_{i_2}\cdots T_{i_{j-1}}(F_{i_j}), \quad 1 \le j \le N.\]
Lusztig showed that the set\[
    \{F_{\beta_1}^{a_1}F_{\beta_2}^{a_2}\cdots F_{\beta_N}^{a_N}~|~a_1,a_2,\ldots,a_N\in \mathbb{Z}_{\geq 0}\}
  \] is a basis of $\mathbf{U}^-$.

Kolb and Pellegrini established the braid group actions for $\mathbf{U}^{\imath}$ as follows:
\begin{prop}[{\cite[Theorem~3.3]{KP11}}]
  There exist $n$ algebra automorphisms $\tau_i$ $(1\leq i\leq n)$ such that
  \[
    \tau_i(B_j)=\begin{cases}
      B_j,&\text{if }a_{ij}=0 \mbox{ or } 2,\\
      B_iB_j-q_iB_jB_i,&\text{if }a_{ij}=-1,\\
      \sum\limits_{s=0}^2(-1)^sq^sB_i^{(2-s)}B_jB_i^{(s)}+B_j,&\text{if }a_{ij}=-2,\\[3pt]
      \sum\limits_{s=0}^3(-1)^sq^sB_i^{(3-s)}B_jB_i^{(s)}+\frac{B_iB_j-q^3B_jB_i}{q[3]!}+B_iB_j-qB_jB_i,&\text{if }a_{ij}=-3,
    \end{cases}
  \]
  for all $1\leq i,j\leq n$.
\end{prop}

 Denote
\[
  B_{\beta_j}=\tau_{i_1}\tau_{i_2}\cdots \tau_{i_{j-1}}(B_{i_j}), \quad 1 \le j \le N.
\]
\begin{prop}[{\cite[Theorem~4.4]{XY14}}]
  The set\[
   \mathbf{B}_{\mathrm{PBW}}= \{B_{\beta_1}^{a_1}B_{\beta_2}^{a_2}\cdots B_{\beta_N}^{a_N}~|~a_1,a_2,\ldots,a_N\in \mathbb{Z}_{\geq 0}\}
  \] forms a basis of $\mathbf{U}^{\imath}$.
\end{prop}

Denote by $\mathcal{GL}(\mathbf{U}^+)$ the set of $\mathbf{U}^+$-good Lyndon words. Recall that the map $\psi:\mathcal{GL}(\mathbf{U}^+) \rightarrow\Phi^+$, given by $l\mapsto |l|$, is a bijection (see \cite[Proposition~18]{Lec04}). For each $\beta \in\Phi^+$, we set $l(\beta)=\psi^{-1}(\beta)$. We also set $\mathrm{ht}(\beta)=\sum_{i=1}^n k_i$ if $\beta=\sum_{i=1}^n k_i\alpha_i$. We now discuss the relationship between $\mathbf{B}_\mathcal{L}$ and $\mathbf{B}_{\mathrm{PBW}}$.

Consider the unique $\mathbb{Q}$-algebra anti-automorphism $\omega$ of $T(V)$ defined by $\omega(q)=q^{-1}$ and $\omega(v_{i_1}v_{i_2}\cdots v_{i_k})=v_{i_k}\cdots v_{i_2}v_{i_1}$. Let $\pi^+$ (resp. $\pi^-$) be the algebra homomorphism from $T(V)$ to $\mathbf{U}^+$ (resp. $\mathbf{U}^-$) that maps $v_i$ to $E_i$ (resp. $F_i$) for each $i$. Then $\tau \pi^+=\pi^-\omega$ and $\tau \pi^-=\pi^+\omega$.

For each $\lambda=\sum_{i=1}^nk_i\alpha_i\in \mathbb{N}\Pi$, denote $$\mathcal{N}(\lambda)=\frac{1}{2}\bigg((\lambda,\lambda)-\sum_{i=1}^nk_i(\alpha_i,\alpha_i)\bigg).$$Given $w\in \mathcal{W}$, we denote $\mathcal{N}(w)=\mathcal{N}(|w|)$ for convenience.

\begin{lem}For any $l\in \mathcal{L}$, we have $$\omega([l])=-q^{-\mathcal{N}(l)}[l].$$\end{lem}
\begin{proof}
We use induction on $\ell(l)$. If $\ell(l)=1$, the result is trivial. We assume that $\ell(l)>1$ and the result holds for Lyndon words with length $<\ell(l)$. Let $\mu(l)=(l_1,l_2)$ be the co-standard factorization of $l=l_1l_2$. Then \begin{align*}
\omega([l])&=\omega\big([l_1][l_2]-q^{(|l_1|,|l_2|)}[l_2][l_1]\big)\\[3pt]
&=\omega([l_2])\omega([l_1])-q^{-(|l_1|,|l_2|)}\omega([l_1])\omega([l_2])\\[3pt]
&=q^{-\mathcal{N}(l_1)-\mathcal{N}(l_2)}[l_2][l_1]-q^{-(|l_1|,|l_2|)-\mathcal{N}(l_1)-\mathcal{N}(l_2)}[l_1][l_2]\\[3pt]
&=-q^{-(|l_1|,|l_2|)-\mathcal{N}(l_1)-\mathcal{N}(l_2)}\big([l_1][l_2]-q^{(|l_1|,|l_2|)}[l_2][l_1]\big)\\[3pt]
&=-q^{-\mathcal{N}(l)}[l],
\end{align*}where the last equality follows from the identity $(|l_1|,|l_2|)+\mathcal{N}(l_1)+\mathcal{N}(l_2)=\mathcal{N}(l)$.\end{proof}

\begin{thm}\label{thm:pbwlyndon}
  For any positive root $\beta \in \Phi^+$, there exists a polynomial $h_{\beta}$ with $\deg h_{\beta}<\mathrm{ht}(\beta)$ and non-zero $\kappa_\beta \in \mathbb{Q}(q)$ such that
  \[
    B_{\beta} = \kappa_{\beta} B_{[l(\beta)]} + h_{\beta}(B_1, \ldots, B_n).
  \]
\end{thm}
\begin{proof}For $1\leq j\leq p$, we write $E_{\beta_j}=T_{i_1}T_{i_2}\cdots T_{i_{j-1}}(E_{i_j})$. Notice that $F_{\beta_j}=\tau(E_{\beta_j})$. It follows from \cite[Theorem~28]{Lec04} that $E_{[l(\beta)]}=\overline{E}_{\{l(\beta)\}}$ is proportional to $\overline{E}_{\beta}$, say $\overline{E}_{\beta}=c_{\beta} E_{[l(\beta)]}$.
Here, $\overline{\ \cdot\ }$ is the bar involution on $\mathbf{U}$, which is the unique $\mathbb{Q}$-algebra automorphism that preserves the generators and maps $q$ to $\overline{q}=q^{-1}$.
Therefore \begin{align*}
F_{\beta}&=\tau\big(c_{\beta} E_{[l(\beta)]}\big)=\overline{c_\beta}\tau\pi^+([l(\beta)])\\
&=\overline{c_\beta}\pi^-\omega([l(\beta)])=\overline{c_\beta}\pi^-\big(q^{-\mathcal{N}(\beta)}[l(\beta)]\big)\\&=\overline{c_\beta}q^{-\mathcal{N}(\beta)}\pi^-\big([l(\beta)]\big)=\overline{c_\beta}q^{-\mathcal{N}(\beta)}F_{[l(\beta)]}\\&
=\kappa_\beta F_{[l(\beta)]},\end{align*}
where $\kappa_\beta=\overline{c_\beta}q^{-\mathcal{N}(\beta)}\neq 0$. Since $\varphi^-$ is an isomorphism of algebras, we have $$\varphi^-(F_{\beta})=\kappa_{\beta}\varphi^-\big( F_{[l(\beta)]}\big)=\kappa_{\beta}B_{[l(\beta)]} +\mathbf{U}^\imath_{\ell(l(\beta))-1}.$$

On the other hand, by \cite[Proposition~4.3]{XY14}, there exist polynomials $f$ and $g$ in $\mathbb{Q}(q)[x_1, \ldots, x_n]$ with $f$ homogeneous and $\deg f > \deg g$ such that
$$F_{\beta}= f(F_1, F_2, \ldots, F_n),$$and $$ B_{\beta}= f(B_1, B_2, \ldots, B_n) + g(B_1, B_2, \ldots, B_n).$$ Then \begin{align*}
\varphi^-(F_{\beta})&=f\big(\varphi^-(F_1), \varphi^-(F_2), \ldots, \varphi^-(F_n)\big)\\&=f(B_1, B_2, \ldots, B_n)+\mathbf{U}^\imath_{\deg f-1}= B_{\beta}+\mathbf{U}^\imath_{\deg f-1}.
\end{align*}
In conclusion, $B_{\beta}-\kappa_{\beta}B_{[l(\beta)]}\in \mathbf{U}^\imath_{\deg f-1}$ as desired.\end{proof}

\subsection{PBW-type bases of $\mathbf{U}^\imath(\mathfrak{sl}_{n+1})$}
In this subsection we keep all conventions used before and assume that $\mathbf{U}^\imath=\mathbf{U}^\imath(\mathfrak{sl}_{n+1})$ is the type $A_n$ split $\imath$quantum group. Rearranging the order of the simple roots if necessary, we may assume that the longest element $w_0$ of $W$ has the following reduced expression
\[
w_0= s_1 s_2 \cdots s_n s_1 s_2 \cdots s_{n-1} \cdots s_1 s_2 s_1.
\]
Sometimes we denote by $s_{i_1}s_{i_2}\cdots s_{i_r}$ a right factor of the above expression with length $r\leq \frac{n(n+1)}{2}$ in the sequel.

Let $C=(c_{ij})$ be the Cartan matrix of type $A_n$. Since $s_i(\alpha_j)=\alpha_j-c_{ij}\alpha_i$, the above reduced expression of $w_0$ provides the following order of positive roots:\[
    \begin{array}{cccc}
        \beta_1 = \alpha_1&< \beta_2 = \alpha_1 + \alpha_2&<\cdots& <\beta_n = \alpha_1 + \alpha_2 + \cdots \alpha_n\\
        &< \beta_{n+1} = \alpha_2&< \cdots& < \beta_{2n-1} = \alpha_2 + \cdots\alpha_n\\
         &\cdots&\cdots&\cdots\\
        & & &< \beta_{\frac{n(n+1)}{2}} = \alpha_n.
    \end{array}
\]By Proposition \ref{Equivalence of two goodnesses} and \cite[8.1]{Lec04},  $$\mathcal{GL}(\mathbf{U}^\imath)=\{v_iv_{i+1}\cdots v_j~|~1\leq i\leq j\leq n\}.$$So the above chain corresponds to that of good Lyndon words with respect to lexicographic order.

Now $\mathbf{U}^\imath$ is generated by $B_1, B_2,\ldots,B_n$ subject to relations
\begin{gather*}
    B_iB_j=B_jB_i,  \text{ if } c_{ij} = 0,\\
    B_i^2B_j - [2] B_iB_jB_i + B_jB_i^2 = -q^{-1} B_j, \text{ if } c_{ij} = -1.
\end{gather*}
The braid group actions $\tau_i$ on $\mathbf{U}^\imath$ are given by:
\[
    \tau_i(B_j) = \begin{cases}
        B_j, &\text{if } c_{ij}=0 \text{ or } 2;\\
        B_iB_j - qB_jB_i,  &\text{if } c_{ij} = -1.
    \end{cases}
\]

We define a bilinear map $[\cdot,\cdot]_{q}:\mathbf{U}^\imath\times \mathbf{U}^\imath\rightarrow \mathbf{U}^\imath$ as follows: for any $P,Q\in \mathbf{U}^\imath$, $$[P, Q]_{q} = PQ - qQP.$$ It is easy to see that this map is well-defined. For $1 \le i < n$, by a direct calculation, we have
\[
    \tau_i(B_{i+1}) = [B_i, B_{i+1}]_{q}, \quad \tau_i\tau_{i+1}(B_i) = B_{i+1}.
\]

\begin{prop}
For any positive root \[
    \beta_r = \alpha_i + \alpha_{i+1} + \cdots + \alpha_{i+k}, \ 1\le i \le i+k \le n,
\]
 we have
    \[
        B_{\beta_r} = \begin{cases}
            [\ldots[[B_i, B_{i+1}]_{q}, B_{i+2}]_q\ldots, B_{i+k}]_q, &\text{ if } k>0,\\
            B_i, &\text{ if } k=0.
        \end{cases}
    \]
\end{prop}
\begin{proof}
    We prove it by induction on the rank $n$.
    The case $A_2$ follows immediately from a simple calculation. Suppose that $n>2$ and the result is true for the type $A_{n-1}$.

    For type $A_n$, consider the prefix $w_{r} = s_{i_1}\cdots s_{i_{r-1}}$ of $w_0$.
    If $r \leq n$, then $$\beta_r = \alpha_1 + \cdots + \alpha_{r-1}=s_1\cdots s_{r-1}(\alpha_r).$$It is easy to check the result for $B_{\beta_r}$ since
    $B_{\beta_r} = \tau_1\cdots \tau_{r-1}(B_r)$.
    %Suppose $\beta_r = s_{i_1}s_{i_2}...s_{i_{r-1}}(\alpha_r)= \alpha_{r_0} + \alpha_{r_0+1} + ... + \alpha_{r_1}$.

    If $r>n$, $w_r = s_1 \cdots s_n s_{i_{n+1}}\cdots s_{i_{r-1}}$.
    Note that $s_{i_{n+1}}\cdots s_{i_{r-1}}$ is a prefix of the longest element $$s_1 s_2 \cdots s_{n-1} s_1 s_2 \cdots s_{n-2} \cdots s_1 s_2 s_1$$of the Weyl group of type $A_{n-1}$, and
    \[
        \gamma_r = s_{i_{n+1}}\cdots s_{i_{r-1}}(\alpha_r)
    \]
    is a positive root of type $A_{n-1}$. Suppose
    \[
        \gamma_r = \alpha_{j} + \cdots + \alpha_{j+k},
    \]for some $j,k$. Then by the inductive hypothesis,
    \[
        \tau_{i_{n+1}}\cdots \tau_{i_{r-1}}(B_r) = \begin{cases}
            [\ldots[[B_j, B_{j+1}]_q, B_{j+2}]_q\ldots, B_{j+k}]_q, &\text{ if } k>0\\
            B_j, &\text{ if } k=0.
        \end{cases}
    \]

    Note that for $1 \le k < n$,
    \[
        s_1s_2 \cdots s_n(\alpha_k) = s_1s_2 \cdots s_ks_{k+1}(\alpha_k) = \alpha_{k+1},
    \]
    and
    \[
        \tau_1\tau_2\cdots\tau_n(B_k) = \tau_1\tau_2\cdots\tau_k \tau_{k+1}(B_k) = \tau_1\tau_2\cdots\tau_{k-1}(B_{k+1}) = B_{k+1}.
    \]
 Thus
    \begin{align*}
        \beta_r &= s_1s_2\cdots s_n(\gamma_r) = \alpha_{j+1} + \cdots + \alpha_{j+k+1},\\
        B_{\beta_r} &= \tau_1\tau_2\cdots \tau_n\tau_{i_{n+1}}\cdots \tau_{i_{r-1}}(B_r) \\
        &= \begin{cases}
            [\ldots[[B_{j+1}, B_{j+2}]_q, B_{j+3}]_q\ldots, B_{j+k+1}]_q, &\text{ if } k>0,\\
            B_{j+1}, &\text{ if } k=0.
        \end{cases}
    \end{align*}
\end{proof}

\begin{cor}\label{cor:lynandpbw}
The bases $\mathbf{B}_{\mathcal{L}}$ and $\mathbf{B}_{\mathrm{PBW}}$ coincide for split $\imath$quantum groups $\mathbf{U}^\imath(\mathfrak{sl}_n)$.\end{cor}
\begin{proof}We adopt all conventions above. For any good Lyndon word $v_iv_{i+1}\cdots v_{i+k}$ with $k\geq 1$, we have \begin{align*}
[v_iv_{i+1}\cdots v_{i+k}]&=[[v_iv_{i+1}\cdots v_{i+k-1}],v_{i+k}]_{q^{-1}}=\cdots\\[3pt]
&= [\ldots[[v_i, v_{i+1}]_{q^{-1}}, v_{i+2}]_{q^{-1}}\ldots, v_{i+k}]_{q^{-1}}.
\end{align*}Note that $q^{-(\alpha_j,\alpha_{j+1})}=q$. So $B_{\beta_r}=B_{[l(\beta_r)]}$ for any positive root $\beta_r$.\end{proof}

\section{Canonical bases of $\imath$quantum groups}

In this section, we denote by $\overline{\ \cdot\ }$ the bar involution on $\mathbf{U}^\imath$ which is the unique $\mathbb{Q}$-algebra automorphism defined on the generators by $\overline{B}_i=B_i$ and $\overline{q}=q^{-1}$. For the existence of $\overline{\ \cdot\ }$, we refer the reader to \cite{BK15}.

\subsection{Integral conditions}

Denote \[
   X= \{(l_1, l_2)~|~ l_1, l_2 \in \mathcal{GL}(\mathbf{U}^\imath),\  l_1<l_2,\  l_1l_2 \notin \mathcal{G}(\mathbf{U}^\imath)\}.
\]

\begin{lem}\label{Non-good factor lemma}Suppose that $l$ is a Lyndon word. If $l$ is not $\mathbf{U}^\imath$-good, then $l$ has a factor of the form $l_1l_2$ for some $(l_1, l_2)\in X $.
\end{lem}
\begin{proof}
    We use induction on the length of $l$.
    Since letters are always $\mathbf{U}^\imath$-good, we assume that $\ell(l) \geq 2$. Suppose $l = v_{i}v_{j}$ with $i<j$. If $l$ is not $\mathbf{U}^\imath$-good, itself is such a requested factor.
    Assume that $\ell(l) > 2$ and the result holds for length $< \ell(l)$.
    Since $l$ is Lyndon, we have a decomposition $l_1l_2 = l$ with $l_1, l_2 \in \mathcal{L}$ and $l_1 < l_2$.
    If both of $l_1$ and $l_2$ are $\mathbf{U}^\imath$-good, then $l_1l_2$ is such a not $\mathbf{U}^\imath$-good factor of $l$.
    If one of $l_1, l_2$ is not $\mathbf{U}^\imath$-good, say $l_1$, then by the inductive hypothesis we are done.
\end{proof}

Before stating our next result, we introduce a condition on $\mathbf{U}^\imath$ which enables us to express certain elements in $\mathbf{U}^\imath$ as linear combinations of $\mathbf{B}_{\mathcal{G}}$ with $\Zq$-coefficients. A word $w$ is said to be \emph{integral} with respect to $\mathbf{B}_{\mathcal{G}}$ if $$B_{w}=\sum_{g\in \mathcal{G}(\mathbf{U}^\imath),g\preceq w}c_{g,w}B_g \quad \mbox{with $c_{g,w}\in \Zq$}.$$ We introduce the integral condition (IC) for the $\imath$quantum group $\mathbf{U}^\imath$ as follows.
\begin{itemize}
\item {\bf Condition (IC)}: \quad All words are integral with respect to $\mathbf{B}_{\mathcal{G}}$.
\end{itemize}

\begin{prop}\label{Integral expression}The split $\imath$quantum group $\mathbf{U}^\imath$ satisfies the condition $(\mathrm{IC})$ if and only if, for every $(l_1, l_2)\in X$, $l_1l_2$ is integral with respect to $\mathbf{B}_{\mathcal{G}}$.\end{prop}
\begin{proof}Suppose that $l_1l_2$ is integral with respect to $\mathbf{B}_{\mathcal{G}}$ for every $(l_1, l_2)\in X$. We use well-founded induction on $(\mathcal{W},\prec)$. If $w$ is a letter, then $B_w$ is $\mathbf{U}^\imath$-good. There is nothing to prove. We assume that $\ell(w)\geq 2$ and every word $\prec w$ is integral with respect to $\mathbf{B}_{\mathcal{G}}$. Write $w=l_1l_2\cdots l_k$ with $l_1,l_2,\dots,l_k\in \mathcal{L}$ and $l_1\geq l_2\geq \cdots\geq l_k$. If each $l_i$ is $\mathbf{U}^\imath$-good, then so is $w$. We are done. If not, say $l_j$ is not $\mathbf{U}^\imath$-good, then by Lemma \ref{Non-good factor lemma}, there exist $x,y\in \mathcal{W}^\ast$ and $(l_j',l_j'')\in X$ such that $l_j=xl_j'l_j''y$. By our assumption, $$B_{l_j'l_j''}=\sum_{g\in \mathcal{G}(\mathbf{U}^\imath),g\prec l_j'l_j''}c_{g,l_j'l_j''}B_g,$$ for some $c_{g,l_j'l_j''}\in \Zq$. Thus $$B_w=\sum_{g\in \mathcal{G}(\mathbf{U}^\imath),g\prec l_j'l_j''}c_{g,l_j'l_j''}B_{l_1\cdots l_{j-1}xgyl_{j+1}\cdots l_k}.$$By Proposition \ref{Basic facts for LR order}, each summand on the right hand side is strictly smaller than $w$ with respect to the order $\prec$. By the inductive hypothesis, all of them are integral with respect to $\mathbf{B}_{\mathcal{G}}$, and hence $w$ is integral.

The reverse statement is trivial.\end{proof}

\begin{rem}The condition (IC) is essential for our proof of existence of canonical bases below. In order to verify it, we just need to consider the set $X$ by the above proposition. Since the set of good Lyndon words is finite, it can be checked case by case.\end{rem}

\begin{example}[Type $A_2$]\label{A2}Set $\xi_i = q_i^{-1}$ in the definition of $B_i$ for $i=1,2$. Then the type $A_2$ split $\imath$quantum group $\mathbf{U}^\imath(\mathfrak{sl}_3)$ is generated by $B_1, B_2$ subject to relations
\begin{gather*}
    B_1^2B_2 - [2] B_1B_2B_1 + B_2B_1^2 = B_2, \\
    B_2^2B_1 - [2] B_2B_1B_2 + B_1B_2^2 = B_1.
\end{gather*}
By Lemma \ref{Complete list of good Lyndon words} and \cite[8.1]{Lec04}, the good Lyndon words are $v_1, v_1v_2, v_2$. So all good words are of the form $$v_2^{k_1}(v_1v_2)^{k_2}v_1^{k_3}\quad \mbox{with $k_1,k_2,k_3\in \mathbb{Z}_{\geq 0}$}.$$ Thus $$X=\{(v_1,v_1v_2),(v_1v_2,v_2)\}.$$Note that $$ B_{v_1v_1v_{2}} = -B_{v_{2}v_1^2} + [2]B_{(v_1v_{2})v_1} + B_{v_{2}}$$ and
        $$B_{v_1v_{2}v_{2}} = -B_{v_{2}^2v_1} + [2]B_{v_{2}(v_1v_{2})} + B_{v_1}.$$Hence, the split $\imath$quantum group $\mathbf{U}^\imath(\mathfrak{sl}_3)$ satisfies the condition (IC).

In fact, we will show in the next subsection that this keeps true for all $\mathbf{U}^\imath(\mathfrak{sl}_n)$.\end{example}

\begin{example}[Type $B_2$]Set $\xi_i = q_i^{-1}$ again. The type $B_2$ split $\imath$quantum group $\mathbf{U}^\imath(\mathfrak{so}_5)$ is generated by $B_1, B_2$ subject to relations
\begin{gather*}
    B_1^2B_2 - [2]_{q^2} B_1B_2B_1 + B_2B_1^2 = B_2, \\
    B_2^3B_1 - [3] B_2^2B_1B_2 + [3]B_2B_1B_2^2 - B_1B_2^3 = [2]^2(B_2B_1 - B_1B_2).
\end{gather*}Here $[2]_{q^2}=\frac{q^4-q^{-4}}{q^2-q^{-2}}$. By Lemma \ref{Complete list of good Lyndon words} and \cite[8.2]{Lec04}, the good Lyndon words are $v_1, v_1v_2, v_1v_2v_2, v_2$. Here we adopt the usual convention of Dynkin diagram of $B_2$ whose arrow is opposite to that in \cite[8.2]{Lec04}. So all good words are of the form $$v_2^{k_1}(v_1v_2 v_2)^{k_2}(v_1 v_2)^{k_3}v_1^{k_4} \quad\mbox{with $k_1,k_2,k_3,k_4\in \mathbb{Z}_{\geq 0}$}.$$ Thus $$X=\{(v_1,v_1v_2),(v_1,v_1v_2 v_2),(v_1v_2,v_1v_2v_2),(v_1v_2v_2,v_2)\}.$$A direct computation shows that\begin{align*}
B_{v_1v_1v_2}&= [2]_{q^2} B_{(v_1v_2)v_1} - B_{v_2v_1^2} + B_2,\\
B_{v_1v_1v_2 v_2}&=[2]_{q^2} B_{(v_1v_2)^2}-[2]_{q^2} B_{v_2(v_1v_2)v_1}+B_{v_2^2v_1^2},\\
B_{v_1v_2v_1v_2v_2}&=-2B_{v_2^3v_1^2}+([3]+[2]_{q^2})B_{v_2^2(v_1v_2) v_1}-[3]B_{v_2(v_1v_2v_2 )v_1}\\
&\ \ \ +[3]B_{(v_1v_2v_2)(v_1v_2)}-[2]_{q^2}B_{v_2(v_1v_2)^2}-[2]_{q^2}[2]^2B_{(v_1v_2)v_1}\\
&\ \ \ +2[2]^2B_{v_2v_1^2}+B_{v_2^3}-[2]^2B_{v_2},\\
B_{v_1v_2v_2,v_2}&=B_{v_2^3v_1}-[3]B_{v_2^2(v_1v_2)}+[3]B_{v_2(v_1v_2v_2)}-[2]^2B_{v_2v_1}+[2]^2B_{v_1v_2}.
\end{align*}So the split $\imath$quantum group $\mathbf{U}^\imath(\mathfrak{so}_5)$ satisfies the condition $(\mathrm{IC})$ as well.\end{example}

\begin{example}[Type $G_2$]\label{ex:G2}
This is a counterexample for the condition $(\mathrm{IC})$. Set $\xi_i = q_i^{-1}$.
The split $\imath$quantum group of type $G_2$ is generated by $B_1, B_2$ subject to the relations
\[
    B_2^2B_1 - [2]_{q^3} B_2B_1B_2 + B_1B_2^2 = B_1,
\]
and
\begin{align*}
   \lefteqn{B_1^4B_2 - [4] B_1^3B_2B_1 + \qbi B_1^2B_2B_1^2 - [4] B_1B_2B_1^3 + B_2B_1^4}\\
    &=  ([3]^2 + 1)(B_1^2B_2 + B_2B_1^2) - [2]([2][4] + q^{-2} + q^2)B_1B_2B_1 - [3]^2 B_2.
\end{align*}Here $[2]_{q^3}=\frac{q^6-q^{-6}}{q^3-q^{-3}}$. According to Lemma \ref{Complete list of good Lyndon words} and \cite[5.5.4]{Lec04}, the good Lyndon words are
\[
    v_1, v_1v_1v_1v_2, v_1v_1v_2, v_1v_1v_2v_1v_2, v_1v_2, v_2.
\]Consider the pair $(v_1v_1v_1v_2,v_1v_2)$. Observe that $(v_1v_1v_1v_2,v_1v_2)\in X$ and \begin{align*}
   \lefteqn{B_{v_1v_1v_1v_2v_1v_2}}\\
    &=  B_{v_2(v_1v_2)v_1^3} + (q^2-1+q^{-2})B_{ (v_1v_2)^2v_1^2}-(q^2-1+q^{-2})B_{(v_1v_1v_2v_1v_2)v_1}\\
    &\ \ \ + \frac{1}{[2]} \qbi B_{(v_1v_1v_2)^2} - (q^2+q^{-2})B_{(v_1v_2)(v_1v_1v_1v_2)}\\
    &\ \ \ + (q^2+q^{-2})B_{ v_2(v_1v_1v_1v_2)v_1} - \frac{1}{[2]}\qbi B_{v_2(v_1v_1v_2)v_1^2}\\
&\ \ \   + (-q^6+q^2+1+q^{-2}-q^{-6})B_{(v_1v_2)^2} + (q^6-q^2-1-q^{-2}+q^{-6})B_{v_2(v_1v_2)v_1}.
\end{align*}The number $\frac{1}{[2]}\qbi=\frac{(q^2+1+q^{-2})(q^2+q^{-2})}{q+q^{-1}}$ does not belong to $\Zq$. So the case of type $G_2$ does not satisfy the condition $(\mathrm{IC})$.\end{example}

%\begin{rem}If $\mathbf{U}^\imath$ satisfies the condition $(\mathrm{IC})$, then $\mathbf{B}_{\mathcal{G}}$ is a canonical basis of $\mathbf{U}^\imath$.\end{rem}
%It maybe the canonical basis of the PBW basis.

\begin{lem}\label{Transition lemma}If $\mathbf{U}^\imath$ satisfies the condition $(\mathrm{IC})$, then for any $ g \in \mathcal{G}(\mathbf{U}^\imath)$, $$B_g = B_{[g]} + \sum_{w \in \mathcal{G}(\mathbf{U}^\imath), w \prec g}r_{w, g}B_{[w]}$$ where $r_{w, g} \in \Zq$.
\end{lem}
\begin{proof}By Lemma \ref{bracket triangular},
    \[
        [g] = g + \sum_{w > g, \ell(w) = \ell(g)} b_{w, g} w,
    \]with $ b_{w, g}\in \Zq$. It follows from Proposition \ref{Integral expression} that \[B_{[g]}=B_g+\sum_{w > g, \ell(w) = \ell(g)} b_{w, g} B_w=B_g+\sum_{w > g, \ell(w) = \ell(g)} \sum_{u\in \mathcal{G}(\mathbf{U}^\imath),u\preceq w}b_{w, g}c_{u,w}B_u,\]where $c_{u,w}\in \Zq$. As a consequence, the transition matrix from the basis $\mathbf{B}_\mathcal{G}$ to $\mathbf{B}_\mathcal{L}$ is upper unitriangular and each entry is a number in $\Zq$. Then its inverse matrix is upper unitriangular and has $\Zq$-entries as well, as desired.\end{proof}

\begin{lem}\label{Comparison lemma}If $w,x,y,z\in \mathcal{W}$ satisfy $w\prec x$ and $y\prec z$, then all of the following words $$wy,wz,xy$$ are strictly smaller than $xz$ with respect to the order $\prec$.

Moreover, if $x,z\in \mathcal{L}$ and $x<z$, then all of the following words $$yx, yw, zw $$ are strictly smaller than $xz$ with respect to $\prec$.\end{lem}
\begin{proof}If $\ell(w)<\ell(x)$ or $\ell(y)<\ell(z)$, then both of $\ell(wy)$ is strictly smaller that $\ell(xz)$, and hence $wy\prec xz$. Assume that $\ell(w)=\ell(x)$ and $\ell(y)=\ell(z)$. By (2) of Proposition \ref{Basic facts for LR order}, we have $wy\prec xz$.

On the other hand, if $\ell(w)<\ell(x)$, then $\ell(wz)<\ell(xz)$. Thus $wz\prec xz$. If $\ell(w)=\ell(x)$, then by Proposition \ref{Basic facts for LR order} again, $wz\prec xz$. The case of $xy$ follows from (1) of Proposition \ref{Basic facts for LR order} immediately.

If $x,z\in \mathcal{L}$ and $x<z$, then $xz\in\mathcal{L}$ by Proposition \ref{Basic facts of Lyndon words}. As a consequence, $$xz<z<zx.$$ By Proposition \ref{Basic facts for LR order}, $$yx\prec zx\prec xz,$$ $$yw\prec zw \prec zx\prec xz,$$and $$zw\prec zx\prec xz.$$\end{proof}

\begin{prop}\label{Inductive property}Suppose that $x,y\in \mathcal{W}$ and $l_1,l_2\in \mathcal{L}$ satisfy $x\prec l_1$, $y\prec l_2$ and $l_1<l_2$. If $\mathbf{U}^\imath$ satisfies the condition $(\mathrm{IC})$, then all of the elements $$B_{[x]}B_{[y]},B_{[y]}B_{[x]}, B_{[x]}B_{[l_2]}, B_{[l_2]}B_{[x]},B_{[l_1]}B_{[y]}, B_{[y]}B_{[l_1]}$$ are in the space spanned by the set $\{B_{[g]}|g\in\mathcal{G}(\mathbf{U}^\imath), g\prec l_1l_2\}$ over $\Zq$.\end{prop}
\begin{proof}By Proposition \ref{bracket triangular for words}, we have $$[x]=x+\sum_{u>x,\ell(u)=\ell(x)}b_{u,x}u \quad \mbox{and} \quad [y]=y+\sum_{w>y,\ell(w)=\ell(y)}b_{w,y}w,$$
where $b_{u,x},b_{w,y}\in \Zq$. Then \begin{align*}
B_{[x]}B_{[y]}&=B_{xy}+\sum_{u>x,\ell(u)=\ell(x)}b_{u,x}B_{uy}+\sum_{w>y,\ell(w)=\ell(y)}b_{w,y}B_{xw}\\
&\ \ \ +\sum_{u>x,\ell(u)=\ell(x)}\sum_{w>y,\ell(w)=\ell(y)}b_{u,x}b_{w,y}B_{u w}\\
&=\sum_{z\prec l_1l_2}a_{z,l_1l_2}B_z,
\end{align*}where $a_{z,l_1l_2}\in\Zq$. Here the last equality follows from the above lemma.

By Proposition \ref{Integral expression} and Lemma \ref{Transition lemma}, we have \begin{align*}
B_{[x][y]}&=\sum_{z\prec l_1l_2}a_{z,l_1l_2}\sum_{g\in \mathcal{G}(\mathbf{U}^\imath),g\preceq z }c_{g,z}B_{g}\\
&=\sum_{z\prec l_1l_2}\sum_{g\in \mathcal{G}(\mathbf{U}^\imath),g\preceq z }a_{z,l_1l_2}c_{g,z}\bigg(B_{[g]} + \sum_{g' \in \mathcal{G}(\mathbf{U}^\imath), g' \prec g}r_{g', g}B_{[g']}\bigg)\\
&\in \mathrm{Span}_{\Zq}\{B_{[g]}|g\in\mathcal{G}(\mathbf{U}^\imath), g\prec l_1l_2\}.
\end{align*}
The other cases can be proven similarly.
\end{proof}

\begin{thm}If $\mathbf{U}^\imath$ satisfies the condition $(\mathrm{IC})$, then for every $l \in \mathcal{L}$, we have $$\overline{B_{[l]}} = B_{[l]} + \sum_{g \prec l,}h_{g, l} B_{[g]},$$ where $h_{g, l} \in \Zq$.
\end{thm}
\begin{proof}By Proposition \ref{bracket triangular for words}, $$[l]=l+\sum_{w>l,\ell(w)=\ell(l)}b_{w,l}w,$$for some $b_{w,l}\in \Zq$. Thus, by Proposition \ref{Integral expression} and Lemma \ref{Transition lemma}, we have\begin{align*}
\overline{B_{[l]}}&=\overline{B_l}+\sum_{w>l,\ell(w)=\ell(l)}\overline{b_{w,l}}\ \overline{B_w}\\
&=B_l+\sum_{w>l,\ell(w)=\ell(l)}\overline{b_{w,l}}B_w\\
&=B_{[l]}-\sum_{w>l,\ell(w)=\ell(l)}b_{w,l}B_w+\sum_{w>l,\ell(w)=\ell(l)}\overline{b_{w,l}}B_w\\
&=B_{[l]}+ \sum_{w>l,\ell(w)=\ell(l)}(\overline{b_{w,l}}-b_{w,l})\sum_{u \in \mathcal{G}(\mathbf{U}^\imath), u \preceq w}c_{g, w}B_{u}\\
&=B_{[l]}+ \sum_{w>l,\ell(w)=\ell(l)}\sum_{u \in \mathcal{G}(\mathbf{U}^\imath), u \preceq w}(\overline{b_{w,l}}-b_{w,l})c_{g, w}\sum_{g \in \mathcal{G}(\mathbf{U}^\imath), g \preceq u}r_{g, u}B_{[g]}\\
&= B_{[l]} + \sum_{g \prec l,}h_{g, l} B_{[g]}.\end{align*}
 Here all coefficients are in $\Zq$.
\end{proof}

\begin{cor}If $\mathbf{U}^\imath$ satisfies the condition $(\mathrm{IC})$, then for any good word $g$, we have $$\overline{B_{[g]}}=B_{[g]}+\sum_{w \in \mathcal{G}(\mathbf{U}^\imath), w \prec g}h_{w,g}B_{[w]},$$where $h_{w,g}\in \Zq$.\end{cor}
\begin{proof}Write $g=g_1g_2\cdots g_k$ with $g_1,g_2,\ldots, g_k\in \mathcal{GL}(\mathbf{U}^\imath)$ and $g_1\geq g_2\geq \cdots \geq g_k$. We use induction on $k$. When $k=1$, it is a direct consequence of the above theorem. Assume that $k\geq 2$ and the result holds for $k-1$. We set $g'=g_2g_3\cdots g_k$. Then $g'$ is a good word and $g=g_1g'$. By inductive hypothesis, $$\overline{B_{[g']}}=B_{[g']}+\sum_{w \in \mathcal{G}(\mathbf{U}^\imath), w \prec g'}h_{w,g'}B_{[w]},$$where $h_{w,g'}\in \Zq$. Thus\begin{align*}
\overline{B_{[g]}}&=\overline{B_{[g_1]}}\ \overline{B_{[g']}}\\[3pt]
&=\bigg(B_{[g_1]}+\sum_{x \in \mathcal{G}(\mathbf{U}^\imath), x \prec g_1}h_{x,g_1}B_{[x]}\bigg)\bigg(B_{[g']}+\sum_{y \in \mathcal{G}(\mathbf{U}^\imath), y \prec g'}h_{y,g'}B_{[y]}\bigg)\\[3pt]
&=B_{[g]}+\sum_{x \in \mathcal{G}(\mathbf{U}^\imath), x \prec g_1}h_{x,g_1}B_{[x]}B_{[g']}+\sum_{y \in \mathcal{G}(\mathbf{U}^\imath), y \prec g'}h_{y,g'}B_{[g_1]}B_{[y]}\\[3pt]
&\ \ \ +\sum_{x \in \mathcal{G}(\mathbf{U}^\imath), x \prec g_1}\sum_{y \in \mathcal{G}(\mathbf{U}^\imath), y \prec g'}h_{x,g_1}h_{y,g'}B_{[x]}B_{[y]}.
\end{align*}By Proposition \ref{Inductive property}, we get the desired result.\end{proof}

\begin{thm}\label{Canonical bases}Suppose that $\mathbf{U}^\imath$ admits the condition $(\mathrm{IC})$. Then there exists a unique basis (called the canonical basis) $\mathbf{B}_{can}=\{C_g~|~g\in \mathcal{G}(\mathbf{U}^\imath)\}$ of $\mathbf{U}^\imath$ such that, for each $g\in \mathcal{G}(\mathbf{U}^\imath)$, \begin{itemize}
\item[(1)] $\overline{C_g}=C_g$;
\item[(2)] $C_g=B_{[g]}+\sum_{w\in \mathcal{G}(\mathbf{U}^\imath),w\prec g}p_{g,w}B_{[w]}$, where $p_{g,w}\in q\mathbb{Z}[q]$.
  \end{itemize}\end{thm}
\begin{proof}It is a direct consequence of the above corollary and the standard procedure for producing canonical bases (see, e.g., \cite[Theorem~0.28]{DDPW08}).\end{proof}

\begin{example}[Canonical basis of type $A_2$]The first few elements of canonical basis of $\mathbf{U}^\imath(\mathfrak{sl}_3)$ are\begin{align*}
    C_{v_1} &= B_{[v_1]}=  B_{v_1}, \\
    C_{v_2v_1} &= B_{[v_2v_1]}=B_{v_2v_1},\\
    C_{v_1v_2} &= B_{[v_1v_2]}+qB_{[v_2v_1]}= B_{ v_1v_2}, \\
    C_{v_2} &= B_{[v_2]}= B_{v_2}.
\end{align*}\end{example}

\begin{example}[Canonical basis of type $B_2$]We adopt the convention in Example \ref{Lyndon basis of B2} for simple roots. The first few elements of canonical basis of $\mathbf{U}^\imath(\mathfrak{so}_5)$ are\begin{align*}
    C_{v_1} &= B_{[v_1]}= B_{v_1}, \\
    C_{v_2} &= B_{[v_2]}=B_{v_2},\\
    C_{v_2v_1} &= B_{[v_2v_1]}= B_{v_2v_1}, \\
    C_{v_1v_2} &= B_{[v_1v_2]}+q^2B_{[v_2v_1]}=B_{v_1v_2}, \\
    C_{v_2v_2v_1} &= B_{[v_2v_2v_1]}= B_{v_2v_2v_1},\\
    C_{v_2v_1v_2} &= B_{ [v_2v_1v_2]}+q^2B_{[v_2v_2v_1]}= B_{ v_2v_1v_2},\\
    C_{v_1v_2v_2} &=B_{ [v_1v_2v_2]}+q^2B_{ [v_2v_1v_2]}+(q^4-q^2)B_{[v_2v_2v_1]}\\
    &=B_{ v_1v_2v_2} -B_{v_2v_1v_2}.
\end{align*}\end{example}

\begin{rem}\label{rem:can}
(1) Our canonical bases of split $\imath$quantum groups are not special cases of the ``dual $\imath$canonical bases'' of universal $\imath$quantum groups established by Lu and Wang in \cite{LW21b}. The discrepancy arises due to different selections of the PBW type bases serving as the standard bases.

Moreover, our canonical bases and Lu-Wang's ``dual $\imath$canonical bases'' are both completely different from the $\imath$canonical bases of modified $\imath$quantum groups investigated by Bao and Wang in \cite{BW18b}. The connection among these three canonical bases remains unclear.

(2) Since the order $\prec$ can also be used in the half quantum group $\mathbf{U}^+$, our discussion is still valid for $\mathbf{U}^+$. In this case, we just ignore the shorter-length terms in the formulas because $\mathbf{U}^+$ is a graded algebra, and the results in this section hold as well.

\end{rem}

\subsection{Canonical bases of $\mathbf{U}^\imath(\mathfrak{sl}_{n+1})$}
In this subsection, we assume that $n\geq 3$ and fix $\xi_i = q_i^{-1}$ in the definition of $B_i$ for all $i$.

The split $\imath$quantum group $\mathbf{U}^\imath=\mathbf{U}^\imath(\mathfrak{sl}_{n+1})$ of type $A_n$ is generated by $B_1, B_2,\ldots,B_n$ subject to relations
\begin{gather*}
B_iB_j=B_jB_i,  \text{ if }|i-j|\geq 2,\\
    B_i^2B_{i+1} - [2] B_iB_{i+1}B_i + B_{i+1}B_i^2 = B_{i+1}, \\
    B_{i+1}^2B_i - [2] B_{i+1}B_iB_{i+1} + B_iB_{i+1}^2 = B_i,
\end{gather*}for $1\leq i<n$.

By Proposition \ref{Equivalence of two goodnesses} and \cite[8.1]{Lec04}, we have $$\mathcal{GL}(\mathbf{U}^\imath)=\{v_iv_{i+1}\cdots v_j|1\leq i\leq j\leq n\}.$$Recall that $$\mathcal{G}(\mathbf{U}^\imath)=\{l_1l_2\cdots l_m|l_1,\ldots,l_m\in \mathcal{GL}(\mathbf{U}^{\imath})\text{ and }l_1\geq \cdots\geq l_m\}.$$

\begin{lem}\label{Good An}For any $i\geq 1$, each of the words
    $$v_iv_iv_{i+1}, v_iv_{i+1}v_{i+1}, v_iv_{i+1}v_{i+2}v_{i+1}, v_iv_{i+1}v_{i}v_{i+1}v_{i+2}$$ is integral with respect to $\mathbf{B}_{\mathcal{G}}$.
\end{lem}
\begin{proof}In order to simplify the notation, we use a vector $x\in T(V)$ to denote the corresponding element $B_x$ in $\mathbf{U}^{\imath}$. Using the defining relations of $\mathbf{U}^\imath$, we have
    \begin{align*}
        v_iv_iv_{i+1} &= -v_{i+1}v_iv_i + [2]v_iv_{i+1}v_i + v_{i+1},\\
        v_iv_{i+1}v_{i+1} &= -v_{i+1}v_{i+1}v_i + [2]v_{i+1}v_iv_{i+1} + v_i,\\
        v_iv_{i+1}v_{i+2}v_{i+1} &= \frac{1}{[2]} (v_iv_{i+1}v_{i+1}v_{i+2} + v_{i+2}v_iv_{i+1}v_{i+1} - v_{i+2}v_i),\\
        &= \frac{1}{[2]} (-v_{i+1}v_{i+1}v_iv_{i+2} + [2]v_{i+1}v_iv_{i+1}v_{i+2} + v_{i+2}v_i) \\
        &\ \ \  + \frac{1}{[2]}(-v_{i+2}v_{i+1}v_{i+1}v_i + [2]v_{i+2}v_{i+1}v_iv_{i+1} + v_{i+2}v_i) - \frac{1}{[2]}v_{i+2}v_i,\\
        &=-\frac{1}{[2]} v_{i+1}v_{i+1}v_{i+2}v_i + v_{i+1}v_iv_{i+1}v_{i+2} \\
        &\ \ \  - \frac{1}{[2]}v_{i+2}v_{i+1}v_{i+1}v_i + v_{i+2}v_{i+1}v_iv_{i+1} + \frac{1}{[2]}v_{i+2}v_i,\\
        &= \frac{1}{[2]} (v_{i+2}v_{i+1}v_{i+1}v_i - [2]v_{i+1}v_{i+2}v_{i+1}v_i - v_{i+2}v_i) + v_{i+1}v_iv_{i+1}v_{i+2} \\
        &\ \ \  - \frac{1}{[2]}v_{i+2}v_{i+1}v_{i+1}v_i + v_{i+2}v_{i+1}v_iv_{i+1} + \frac{1}{[2]}v_{i+2}v_i,\\
        &= - v_{i+1}v_{i+2}v_{i+1}v_i + v_{i+1}v_iv_{i+1}v_{i+2} + v_{i+2}v_{i+1}v_iv_{i+1},
\end{align*}
\begin{align*}
        v_iv_{i+1}v_iv_{i+1}v_{i+2} &= \frac{1}{[2]}(v_iv_iv_{i+1}v_{i+1}v_{i+2} + v_iv_{i+1}v_{i+1}v_{i+2}v_i - v_{i+2}v_iv_i)\\
        &= \frac{1}{[2]}(-v_iv_iv_{i+2}v_{i+1}v_{i+1} + [2]v_iv_iv_{i+1}v_{i+2}v_{i+1} + v_{i+2}v_iv_i) \\
        &\ \ \ + \frac{1}{[2]}(-v_{i+1}v_{i+1}v_{i+2}v_iv_i + [2]v_{i+1}v_iv_{i+1}v_{i+2}v_i + v_{i+2}v_iv_i)\\
        &\ \ \  - \frac{1}{[2]}v_{i+2}v_iv_i\\
        &= -\frac{1}{[2]}v_{i+2}v_iv_iv_{i+1}v_{i+1} + v_iv_iv_{i+1}v_{i+2}v_{i+1} \\
        &\ \ \ - \frac{1}{[2]}v_{i+1}v_{i+1}v_{i+2}v_iv_i + v_{i+1}v_iv_{i+1}v_{i+2}v_i + \frac{1}{[2]}v_{i+2}v_iv_i\\
        &= -\frac{1}{[2]}v_{i+2}v_iv_iv_{i+1}v_{i+1} + v_iv_iv_{i+1}v_{i+2}v_{i+1} \\
        &\ \ \ +\frac{1}{[2]}v_{i+2}v_{i+1}v_{i+1}v_iv_i -v_{i+1}v_{i+2}v_{i+1}v_iv_i + v_{i+1}v_iv_{i+1}v_{i+2}v_i\\
        &= -\frac{1}{[2]}v_{i+2}v_iv_iv_{i+1}v_{i+1} \\
        &\ \ \ + v_{i+1}v_{i+2}v_{i+1} v_iv_i - [2] v_{i+1}v_{i+2}v_iv_{i+1}v_i + [2]v_iv_{i+1}v_{i+2}v_iv_{i+1}\\
        &\ \ \ +\frac{1}{[2]}v_{i+2}v_{i+1}v_{i+1}v_iv_i -v_{i+1}v_{i+2}v_{i+1}v_iv_i + v_{i+1}v_iv_{i+1}v_{i+2}v_i\\
        &= -\frac{1}{[2]}v_{i+2}v_{i+1}v_{i+1}v_iv_i + v_{i+2}v_{i+1}v_iv_{i+1}v_i - v_{i+2}v_iv_{i+1}v_iv_{i+1}\\
        &\ \ \ + v_{i+1}v_{i+2}v_{i+1} v_iv_i - [2] v_{i+1}v_{i+2}v_iv_{i+1}v_i + [2]v_iv_{i+1}v_{i+2}v_iv_{i+1}\\
        &\ \ \ + \frac{1}{[2]}v_{i+2}v_{i+1}v_{i+1}v_iv_i -v_{i+1}v_{i+2}v_{i+1}v_iv_i + v_{i+1}v_iv_{i+1}v_{i+2}v_i\\
        &= v_{i+2}v_{i+1}v_iv_{i+1}v_i - v_{i+2}v_iv_{i+1}v_iv_{i+1}\\
        &\ \ \  - [2] v_{i+1}v_{i+2}v_iv_{i+1}v_i + [2]v_iv_{i+1}v_{i+2}v_iv_{i+1} + v_{i+1}v_iv_{i+1}v_{i+2}v_i.
    \end{align*}
\end{proof}

For any $1\leq i<j\leq n$, we set $v[i, j] = v_iv_{i+1}\cdots v_{j-1}$. We consider the set\begin{align*}
Y& =\{(l_1, l_2)| l_1, l_2 \in \mathcal{GL}(\mathbf{U}^\imath),\  l_1<l_2\}\\
    &=\{(v[i, j], v[r, s])|i < r\} \cup \{(v[r, j], v[r, s])|  j< s\},
\end{align*}which contains $X$.

\begin{prop}For any $(v[i, j], v[r, s])\in Y$, $v[i, j]v[r, s]$ is integral with respect to $\mathbf{B}_\mathcal{G}$.
\end{prop}
\begin{proof}Set $$\mathcal{Y}_k=\{l_1l_2|(l_1, l_2)\in Y\text{ and }\ell(l_1l_2)=k\}$$ for each $k\geq 2$, and let $\mathcal{Y}=\bigcup_{k\geq 2}\mathcal{Y}_k$. We use well-founded induction on $(\mathcal{Y},\prec)$.

We first observe that $v[i,j]v[r,s]$ is integral with respect to $\mathbf{B}_\mathcal{G}$ if $j\leq r$. Indeed, if $j=r$, then $v[i,j]v[j,s]=v[i,s]\in \mathcal{G}(\mathbf{U}^\imath)$. If $j<r$, then $B_{v[i,j]}B_{v[r,s]}=B_{v[r,s]}B_{v[i,j]}$ and $v[r,s]v[i,j]\in \mathcal{G}(\mathbf{U}^\imath)$.

We verify that the result holds for $\mathcal{Y}_2$ and $\mathcal{Y}_3$. The case of $\mathcal{Y}_2$ is a direct consequence of the above observation. Note that $$\mathcal{Y}_3=\{v_iv_rv_{r+1}, v_iv_{i+1}v_r, v_iv_iv_{i+1}|i<r\}.$$ By the above observation and Lemma \ref{Good An}, every element in $\mathcal{Y}_3$ is integral with respect to $\mathbf{B}_\mathcal{G}$.

Suppose that $v[i, j]v[r, s]\in \mathcal{Y}_m$ with $m> 3$ and the result holds for all elements $\prec v[i, j]v[r, s]$. We also assume that $j > r$. In the following discussion, we adopt the convention used before which abbreviates $B_x$ to $x$. If $i < r$ and $j \ge r+2$, then
    \begin{align*}
        v[i, j]v[r, s] &=v[i, r-1] v[r-1, r+2] v[r+2, j]v_r v[r+1, s] \\
        &= v[i, r-1] v[r-1, r+2] v_r v[r+2, j] v[r+1, s].
    \end{align*}
    By the previous lemma,
    \begin{align*}
        v[r-1, r+2]v_r = v_{r-1}v_rv_{r+1}v_r=  - v_{r}v_{r+1}v_{r}v_{r-1} + v_{r}v_{r-1}v_{r}v_{r+1} + v_{r+1}v_{r}v_{r-1}v_{r}.
    \end{align*}
    Then
    \begin{align*}
        v[i, j]v[r, s] &=-v[i, r-1]v_{r}v_{r+1}v_{r}v_{r-1}v[r+2, j]v[r+1, s]\\
        &\ \ \  + v[i, r-1]v_{r}v_{r-1}v_{r}v_{r+1}v[r+2, j]v[r+1, s] \\
        &\ \ \  + v[i, r-1]v_{r+1}v_{r}v_{r-1}v_{r}v[r+2, j]v[r+1, s].
    \end{align*}
   All terms on the right hand side are $\prec v[i, j]v[r, s]$. By the inductive hypothesis, we are done.

    If $i < r$ and $j= r+1$, then $$v[i, j]v[r, s] = v[i, r-1]v[r-1, r+1]v_rv[r+1, s].$$ Note that $v[r-1, r+1]v_r = v_{r-1}v_rv_r$ and work like before.

 The next step is to treat the elements $v[r, j]v[r, s]$ with $j< s$.
    If $r<j\leq  r+2$, we note that $$v[r, r+1]v[r, s] = v_rv_rv_{r+1}v[r+2, s]$$ and
$$v[r, r+2]v[r, s] = v_rv_{r+1}v_rv_{r+1}v_{r+2}v[r+3, s].$$
    These two cases can be dealt with as before by using Lemma \ref{Good An}.

    If $j \ge r+3$,
    \begin{align*}
        v[r, j]v[r, s] &= v_r v_{r+1} v[r+2, j] v_r v[r+1, s]\\
        &= v_r v_{r+1} v_r v[r+2, j] v[r+1, s]\\
        &= \frac{1}{[2]}(v_r v_r v_{r+1} + v_{r+1} v_r v_r - v_{r+1}) v[r+2, j] v[r+1, s]\\
        &= \frac{1}{[2]} v_r v_r v[r+1, j] v[r+1, s]\\
        &+ \frac{1}{[2]} v_{r+1} v[r+2, j] v_r v_r v_{r+1} v[r+2, s]
        - \frac{1}{[2]} v_{r+1} v[r+2, j] v[r+1, s]\\
        &= \frac{1}{[2]} v_r v_r v[r+1, j] v[r+1, s]\\
        &- \frac{1}{[2]} v[r+1, j] v[r+1, s] v_r v_r
        + v[r+1, j] v[r, s] v_r\\
        &= \frac{1}{[2]} v_r v_r v[r+1, j] v[r+1, j+1]v[j+1, s]\\
        &- \frac{1}{[2]} v[r+1, j] v[r+1, j+1]v[j+1, s] v_r v_r
        + v[r+1, j] v[r, s] v_r.
    \end{align*}
    By the inductive hypothesis, $v[r+1, j] v[r+1, j+1]$ can be written as combinations of good words with $\Zq$ coefficients.

    We claim that in the good word representation of $v[i, j] v[i, j+1]$ ($i \ge r+1$),
    the multiplicity of a letter $v_k$ appearing in each term of the representation is the same as that of $v_k$ appearing in $v[i, j] v[i, j+1]$.
    We can see it by induction on $j-i$. Small cases $j = i+1, i+2$ have been checked before.
    For general case, we note that a good word in type A which is composed of letters of $v[i, j] v[i, j+1]$ must be of the form:
    \[
        v[j, j+1]^{r} \cdots v[i, i+2]^{s} v[i, i+1]^{t},\ \text{ for } s, t, \cdots, r \ge 0.
    \]
   On the other hand,
    \begin{align*}
        v_{i-1}v_{i-1}v_iv_i &= -v_iv_{i-1}v_{i-1}v_i + [2] v_{i-1}v_iv_{i-1}v_i + v_iv_i\\
        &= v_iv_iv_{i-1}v_{i-1} - [2] v_iv_{i-1}v_iv_{i-1} + [2] v_{i-1}v_iv_{i-1}v_i,\\
        v_{i-1}v_{i-1}v_iv_{i+1}v_i &= - v_iv_{i-1}v_{i-1}v_{i+1}v_i +  [2]v_{i-1}v_iv_{i-1}v_{i+1}v_i + v_iv_{i+1}v_i\\
        &= - v_iv_{i+1}v_{i-1}v_{i-1}v_i +  [2]v_{i-1}v_iv_{i+1}v_{i-1}v_i + v_iv_{i+1}v_i\\
        &= v_iv_{i+1}v_iv_{i-1}v_{i-1} - [2] v_iv_{i+1}v_{i-1}v_iv_{i-1} +  [2]v_{i-1}v_iv_{i+1}v_{i-1}v_i.
    \end{align*}
    Thus,
    \begin{align*}
        &v_{i-1}v_{i-1} v[j, j+1]^{r} \cdots v[i, i+2]^{s} v[i, i+1]^{t}\\
        &= \cdots v_{i-1}v_{i-1} v_i \cdots v_i \cdots\\
        &= \cdots  v_i \cdots v_i \cdots v_{i-1}v_{i-1} - [2]\cdots  v_i \cdots v_{i-1}v_i \cdots v_{i-1} + [2]\cdots  v_{i-1}v_i \cdots v_{i-1}v_i \cdots.
    \end{align*}
Hence
    \begin{align*}
       \lefteqn{ v[r, j]v[r, s]}\\
        &= \frac{1}{[2]} v_r v_r v[r+1, j] v[r+1, j+1]v[j+1, s]\\
        &\ \ \ - \frac{1}{[2]} v[r+1, j] v[r+1, j+1]v[j+1, s] v_r v_r
        + v[r+1, j] v[r, s] v_r \\
        &= \frac{1}{[2]} \Big(v[r+1, j] v[r+1, j+1]v_r v_r+ [2](\text{ integral terms }\prec v[r, j]v[r, j+1])\Big) v[j+1, s]\\
        &\ \ \ - \frac{1}{[2]} v[r+1, j] v[r+1, j+1]v[j+1, s] v_r v_r+ v[r+1, j] v[r, s] v_r \\
        &= v[r+1, j] v[r, s] v_r\\
        &\ \ \  + \text{integral terms }\prec v[r, j]v[r, j+1] v[j+1, s].
    \end{align*}
    It satisfies the inductive hypothesis.
\end{proof}

Thanks to Proposition~\ref{Integral expression} and Theorem~\ref{Canonical bases}, we get the following theorem.
\begin{thm}\label{canofA}
The split $\imath$quantum group $\mathbf{U}^\imath(\mathfrak{sl}_n)$ satisfies the condition $(IC)$, and hence affords a canonical basis. \end{thm}

%%%%%%%%%
%%%%%%%%%

\end{document}